\documentclass[11pt]{amsart}
\usepackage{amssymb}
\usepackage[margin=1in]{geometry}
\usepackage[colorlinks]{hyperref}
\newtheorem{theorem}{Theorem}[section]

\newtheorem{prop}[theorem]{Proposition}
\newtheorem{cor}[theorem]{Corollary}

\theoremstyle{definition}
\newtheorem{definition}[theorem]{Definition}

\theoremstyle{remark}
\newtheorem{remark}[theorem]{Remark}
\newtheorem{notation}[theorem]{Notation}
\numberwithin{equation}{section}
\begin{document}

\newcommand{\spacing}[1]{\renewcommand{\baselinestretch}{#1}\large\normalsize}
\spacing{1.14}

\title{On the Existence of Homogeneous Geodesics in Homogeneous Kropina Spaces}

\author {M. Hosseini}

\address{Department of Mathematics\\ Faculty of  Sciences\\ University of Isfahan\\ Isfahan\\ 81746-73441-Iran.} \email{hoseini\_masomeh@ymail.com}

\author {H. R. Salimi Moghaddam}

\address{Department of Mathematics\\ Faculty of  Sciences\\ University of Isfahan\\ Isfahan\\ 81746-73441-Iran.} \email{hr.salimi@sci.ui.ac.ir and salimi.moghaddam@gmail.com}



\begin{abstract}
Recently, it is shown that each regular homogeneous Finsler space $M$ admits at least one homogeneous geodesic through any point $o\in M$. The purpose of this article is to study the existence of homogeneous geodesics on singular homogeneous $(\alpha,\beta)$-spaces, specially, homogeneous Kropina spaces. We show that any homogeneous Kropina space admits at least one homogeneous geodesic through any point. It is shown that, under some conditions, the same result is true for any $(\alpha,\beta)$-homogeneous space. Also, in the case of homogeneous Kropina space of Douglas type, a necessary and sufficient condition for a vector to be a geodesic vector is given. Finally, as an example, homogeneous geodesics of $3$-dimensional non-unimodular real Lie groups equipped with a left invariant Randers metric of Douglas type are investigated.
\end{abstract}

\maketitle

\section{\textbf{Introduction}}
A geodesic $\gamma : \mathbb{R} \longrightarrow M$ of a Finslerian manifold $(M,F)$ is called a homogeneous geodesic if there exists a one-parameter group of isometries $\phi :\mathbb{R}\times M \longrightarrow M$ such that
\begin{equation*}
\gamma (t)=\phi (t,\gamma (0)), \qquad t\in \mathbb{R}.
\end{equation*}
The problem of the existence of homogeneous geodesics on homogeneous manifolds is an interesting and relatively old problem in differential geometry (see \cite{Dusek} and the references therein). In \cite{Kajzer}, in the case of Lie groups equipped with left invariant metrics, Kajzer showed that there exists at least one homogeneous geodesic passing the identity element. Szenthe, in \cite{Szenthe-1}, proved that if the Lie group is also compact, connected, semi-simple and of rank greater than one, then there are infinitely many homogeneous geodesics through the identity.
One year later, in 2001 (see \cite{Szenthe-2}), he generalized Kajzer's result for left invariant Lagrangian on compact connected Lie groups. Also in the case of compact connected semi-simple Lie groups of rank greater than one, he proved that there are infinitely many homogeneous geodesics for left invariant Lagrangian. In \cite{Kowalski-Szenthe}, Kowalski and Szenthe extended Kajzer's result to homogeneous Riemannian manifolds. They also proved that a homogeneous Riemannian manifold $M = G/H$ with semi-simple group $G$ admits $m = dim(M)$ mutually orthogonal homogeneous geodesics.\\
On the other hand, homogeneous geodesics have found many applications in mechanics. They are usually called relative
equilibria or stationary geodesic in physics literature (for more details see \cite{Arnold,Chossat-Lewis-Ortega-Ratiu,Crampina-Mestdaga,Lacomba,Ortega-Ratiu,Patrick,Toth}). For example in the cases of Lagrangian and Hamiltonian systems, T\'{oth} studied trajectories which are orbits of a one-parameter symmetry group $G$. He obtained criteria under which an orbit of a one-parameter subgroup of a symmetry group is a solution of the Euler-Lagrange or Hamiltonian equations (see \cite{Toth}). As another example, see \cite{Lacomba}, where Lacomba has used homogeneous geodesics in the framework of Smale's mechanical systems with symmetry.\\
Recently, homogeneous geodesics on homogeneous Finsler spaces have been studied by many researchers. In the case of homogeneous Finsler space, a criterion for a nonzero vector to be a geodesic vector, was given by Latifi in \cite{Latifi}. Yan and Deng \cite{Yan-Deng} investigated the existence of homogeneous geodesics on Randers spaces. They proved that  every homogeneous Randers space admits at least a homogeneous geodesic through any point.
In 2017, Yan extended this result to any  homogeneous Finsler space of odd dimension (see\cite{Yan}).\\
Recently, in \cite{Yan-Huang}, Yan and Huang proved that any regular homogeneous Finsler spaces admits at least one homogeneous geodesic through each point. They have used Legendre transformation, which is a bijection for regular Finsler metrics, to find nonzero geodesic vectors. \\
The family of $(\alpha,\beta)$-metrics is an interesting class of Finsler metrics which have many applications in physics. In this paper we study the problem of existence of homogeneous geodesics on singular homogeneous $(\alpha,\beta)$-spaces, specially, homogeneous Kropina spaces. It is shown that, there exists at least one homogeneous geodesic through any point of an arbitrary homogeneous Kropina space. We prove that, under some conditions, the same result is true for any $(\alpha,\beta)$-homogeneous space. In the case of homogeneous Kropina space of Douglas type, we show that a nonzero vector is a geodesic vector of the Kropina metric if and only if it is a geodesic vector of it's base Riemannian metric. Homogeneous geodesics of $3$-dimensional non-unimodular real Lie groups equipped with left invariant Randers metrics of Douglas type are investigated at the end of the paper.


\section{\textbf{Preliminaries}}
This section contains some preliminaries about Finsler spaces. As Chern mentioned in the title of his article \cite{Chern}, "Finsler geometry is just Riemannian geometry without the quadratic restriction". In fact, in 1854, Riemann has introduced a metric as a positive (when the vector $y$ is not equal to zero) function on the tangent bundle which is homogeneous of degree one in $y$. Finsler worked on this general case in his doctoral thesis (see \cite{Finsler}). But Finsler's name established in differential geometry with the book \cite{Cartan} written by Cartan.\\
Consider a differentiable manifold $M$ and denote its tangent bundle by $TM$. A Finsler metric on $M$ is a smooth function $F: TM\setminus \{0\} \longrightarrow \mathbb{R} ^{>0}$ such that
\begin{itemize}
\item[i)] $F(x,\lambda y)=\lambda F(x,y)$, \ \ \ \ for all $\quad  \lambda>0$,
\item[ii)] The hessian matrix $(g_{ij})=\left( \dfrac{1}{2} \dfrac{\partial ^2 F^2}{\partial y^i \partial y^j}\right) $ be positive definite for all $(x,y) \in TM\backslash \{0\}$.
\end{itemize}
An important family of Finsler metrics is  the class of $(\alpha , \beta)$-metrics which arises from Riemannian metrics and 1-forms (see \cite{Matsumoto}). More precisely, $(\alpha , \beta)$-metrics can be expressed in the form  $F = \alpha \phi (\frac{\beta}{\alpha})$, where $\phi : (-b_0 , b_0) \longrightarrow \mathbb{R^+}$ is a $C^\infty$ function and  $\alpha (x,y)=\sqrt{\textbf{\emph{a}}_{ij}y^iy^j}$ and  $\beta (x,y)=b_iy^i$, and $\textbf{\emph{a}}$ and $\beta$ are a Riemannian metric and a 1-form respectively, as follows
\begin{eqnarray}
 && \textbf{\emph{a}}=\textbf{\emph{a}}_{ij }dx^i \otimes  dx^j, \\
 &&\beta =b_i dx^i.
\end{eqnarray}
It can be shown that, $F$ is a regular Finsler metric if $\phi $ satisfying
\begin{equation}\label{alpha-beta metric condition}
\phi (s) - s \phi ^{'} (s) + ( b^2 - s^2 ) \phi ^{''} (s) > 0, \qquad  \vert s \vert  \leq b < b_0,
\end{equation}
and  $\Vert \beta _x\Vert _{\alpha}=\sqrt{\textbf{\emph{a}}^{ij} (x)b_i(x) b_j(x) } < b_0$ for any $x \in M$, where  $(\textbf{\emph{a}}^{ij}(x))$ is the inverse matrix of $(\textbf{\emph{a}}_{ij}(x))$ (for more details see  \cite{Chern-Shen}).\\
In the above definition, if $\phi$ does not satisfy the condition \ref{alpha-beta metric condition} or $\phi(0)$ is not defined, then the $(\alpha , \beta)$-metric is called a singular finsler metric. \\
For instance, if $\phi (s) = 1+s$ or $\phi (s) = \frac{1}{s}$, then we obtain two important classes of Finsler metrics,  called Randers metrics $F=\alpha + \beta $ (which are regular)  and Kropina metrics $F=\frac{\alpha ^2}{\beta}$ (which are singular), respectively. \\
Easily we can see the Riemannian metric $\textbf{\emph{a}}$ induces an inner product on any cotangent
space $T^\ast_xM$. This inner product induces a linear isomorphism between $T^\ast_xM$  and $T_xM$ (see \cite{Deng-Hou-2}). Using this isomorphism, the 1-form $\beta$ corresponds to a vector field $\tilde{X}$ on $M$ such that
\begin{equation}
\textbf{\emph{a}}(y,\tilde{X}(x)) = \beta (x,y).
\end{equation}
Now we can rewrite the Randers and Kropina metrics as follows
\begin{equation}\label{Randers}
F(x,y)=\sqrt{\textbf{\emph{a}}\left( y ,y \right)}+\textbf{\emph{a}}\left(\tilde{X}\left( x\right) ,y\right) ,
\end{equation}
\begin{equation}\label{Kropina}
F(x,y)=\dfrac{\textbf{\emph{a}}\left( y ,y\right)}{\textbf{\emph{a}}\left(\tilde{X}\left( x\right) ,y\right)}.
\end{equation}
Recently, Kropina metrics are known as  a singular solution of the Zermelo's navigation problem on some Riemannian manifold $(M,h)$ under the influence of a vector field $W$ with $\Vert W\Vert _h=1$ . The pair $(h,W)$ is called the navigation data of Kropina metric. In fact, there is one to one correspondence between Kropina metric $F$ and the navigation data $(h,W)$ (see \cite{Zhang-Shen}). In the present paper, we focus on a class of Kropina metrics which are obtained from the initial Zermelo's navigation problem in terms of a Riemannian metric $h$ and a unit vector field $W$. \\
Homogeneous Finsler spaces are defined similar to the Riemannian case. A homogeneous Finsler space is a Finslerian manifold $(M,F)$ on which its isometry group $I(M,F)$ acts transitively on $M$. So, a connected homogeneous Finsler space $M$ can be considered of the form $M=G/H$ where $G$ is a connected Lie group of isometries of $M$, acting transitively on $M$, and $H$ is the isotropy subgroup of a point in $M$.
A homogeneous space $M=G/H$ is said to be  reductive if there exists an $Ad(H)$-invariant decomposition $\mathfrak{g}=\mathfrak{m}+\mathfrak{h}$, where $\frak{g}$ and $\frak{h}$ denote the Lie algebras of $G$ and $H$, respectively and $+$ is the direct sum of subspaces.
It is well known that any homogeneous Riemannian manifold is a reductive homogeneous space (see \cite{Kobayashi-Nomizu,Kowalski-Szenthe}). The same proposition is true for any homogeneous Finsler space (see \cite{Latifi}).\\
Now we study the relation between the isometry group of Kropina metrics and the isometry group the Riemannian metric $h$.
\begin{prop}
Suppose that $(M,F)$ is a Kropina space which arises from a navigation data $(h,W)$. Then the isometry group of $(M,F)$ is a closed subgroup of the isometry group of Riemannian manifold $(M,h)$.
\end{prop}
\begin{proof}
If $(h,W)$ is the navigation data of $F$, then $\Vert W\Vert _h=1$. We put $X=\dfrac{1}{2}W$ and consider the Randers metric $\tilde{F}$ which arises from the Riemannian metric $h$ and the vector field $X$ on $M$.  We claim that $F$ and $\tilde{F}$ have the same isometry group. It suffices to show that $\phi \in I(M,F)$ if and only if $\phi \in I(M,\tilde{F})$. Let $\phi \in I(M,F)$, then according to lemma $4$ of \cite{Yoshikawa-Sabau}, $\phi $ is an isometry of $h$ which preserves $W$ and so $X$ is $\phi$-invariant. It follows from proposition $3.2$ of \cite{Salimi-Results}, $\phi$ belongs to the isometry group of $(M,\tilde{F})$. Conversely, if  $\phi \in I(M,\tilde{F})$, then using proposition $7.1$ of \cite{Deng} and  proposition $3.2$ of \cite{Salimi-Results}, $\phi \in I(M,h)$ and $W$ is $\phi$-invariant. Again, with  lemma $4$ of \cite{Yoshikawa-Sabau}, $\phi \in I(M,F)$.
This finally implies by Proposition  $7.1$ of \cite{Deng} that $ I(M,F)$ is a closed subgroup of the isometry group of the Riemannian manifold $(M,h)$.
\end{proof}
\begin{remark}
According to the previous proposition, $I(M,F)=I(M,\tilde{F})$. This implies that every such a homogeneous Kropina space is reductive.
\end{remark}  
As in the Riemannian homogeneous spaces the definition of homogeneous geodesic can be extended to homogeneous Finsler spaces as follows:
\begin{definition}
Let $(M=G/H,F)$ be a homogeneous Finsler space. A non-zero vector $X\in \mathfrak{g}$ is said to be a geodesic vector if the curve $\gamma (t)=\exp tX .o$ is a geodesic of $(M=G/H,F)$.
\end{definition}

Consider a Riemannian homogeneous space $(M=G/H,\textbf{\emph{a}})$ with a reductive decomposition $\mathfrak{g}=\mathfrak{m}+\mathfrak{h}$. In \cite{Kowalski-Vanhecke}, Kowalski and Vanhecke proved that $X\in \mathfrak{g} \setminus \{0\}$ is a geodesic vector if and only if
 \begin{equation}
 \textbf{\emph{a}}([X,Y]_{\mathfrak{m}}, X_{\mathfrak{m}})=0, \qquad \forall Y \in \mathfrak{m},
\end{equation}
where the subscript $\mathfrak{m}$ indicates the projection into the subspace $\mathfrak{m}$.\\
This proposition is generalized to homogeneous Finsler spaces by Latifi as follows (see \cite{Latifi}):
\begin{theorem}
A nonzero $X\in \mathfrak{g}$ is a geodesic vector if and only if
\begin{equation}
 g_{X_{\mathfrak{m}}}([X,Y]_{\mathfrak{m}}, X_{\mathfrak{m}})=0, \qquad \forall Y \in \mathfrak{g},
 \end{equation}
where $g$ denotes the fundamental tensor of $F$ on $\mathfrak{m}$.
\end{theorem}
As a consequence he proved the following corollary.
\begin{cor}
A nonzero $X\in \mathfrak{g}$ is a geodesic vector if and only if
\begin{equation}
 g_{X_{\mathfrak{m}}}([X,Y]_{\mathfrak{m}}, X_{\mathfrak{m}})=0,\qquad \forall Y \in \mathfrak{m}.
 \end{equation}
\end{cor}
\begin{notation}
For simplicity, from now on we use the notation $\langle,\rangle$ for the Riemannian metric $\textbf{\emph{a}}$.
\end{notation}

\section{\textbf{Homogeneous geodesic in homogeneous Kropina spaces}}
As already mentioned in the introduction, Kowalsky and Szenthe showed that any homogeneous Riemannian space admits at least one homogeneous geodesic on each origin point (see \cite{Kowalski-Szenthe}). Yan and S. Deng generalized this result to homogeneous Randers spaces in \cite{Yan-Deng}. In this section we show that the same result is true for homogeneous Kropina spaces. Firstly, we give a necessary and sufficient condition for a nonzero vector in a homogeneous Kropina spaces to be a geodesic vector.

\begin{prop}\label{geodesic vector of Kropina}
Suppose that $(G/H,F)$ is a homogeneous Finsler space and $F$ is a Kropina metric arising from an invariant Riemannian metric $\left\langle ,\right\rangle $ and an invariant vector field $\tilde{X}$ such that $X=\tilde{X}(H)$. Then, a nonzero vector $Y\in \mathfrak{g}$ is a geodesic vector if and only if
\begin{eqnarray} \label{equivalent of geodesic lemma}
\left\langle [Y,Z]_{\mathfrak{m}},\dfrac{2}{F(Y_{\mathfrak{m}})}Y_{\mathfrak{m}}-X\right\rangle =0
\end{eqnarray}
holds for every $Z\in \mathfrak{m}$.
\end{prop}
\begin{proof}
A nonzero vector $Y\in \mathfrak{g}$ is a geodesic vector if and only if
\begin{eqnarray*}
g_{Y_{\mathfrak{m}}}( [Y,Z]_{\mathfrak{m}},Y_{\mathfrak{m}}) =0 \qquad \forall Z\in \mathfrak{m}.
\end{eqnarray*}
Using the formula (2.5) in \cite{Salimi} we have
\begin{align*}
g_{Y_\frak{m}}([Y,Z]_\frak{m},Y_\frak{m}) =& \dfrac{1}{\left\langle X,Y_\frak{m}\right\rangle ^4}\{  \left( 2\left\langle Y_\frak{m},[Y,Z]_\frak{m}\right\rangle \left\langle Y_\frak{m},X\right\rangle-\left\langle [Y,Z]_\frak{m},X\right\rangle \left\langle Y_\frak{m},Y_\frak{m}\right\rangle \right)\\
&\left( 2\left\langle Y_\frak{m},Y_\frak{m}\right\rangle \left\langle Y_\frak{m},X\right\rangle -\left\langle Y_\frak{m},X\right\rangle \left\langle Y_\frak{m},Y_\frak{m}\right\rangle \right) + \left\langle Y_\frak{m},Y_\frak{m}\right\rangle \left\langle Y_\frak{m},X\right\rangle \\
&\left( 2\left\langle [Y,Z]_\frak{m},Y_\frak{m}\right\rangle \left\langle Y_\frak{m},X\right\rangle +2\left\langle Y_\frak{m},Y_\frak{m}\right\rangle \left\langle [Y,Z]_\frak{m},X\right\rangle -2\left\langle Y_\frak{m},X\right\rangle ,\left\langle Y_\frak{m},[Y,Z]_\frak{m}\right\rangle  \right)\\
 &-2\left\langle Y_\frak{m},Y_\frak{m}\right\rangle \left\langle [Y,Z]_\frak{m},X\right\rangle \left( 2\left\langle Y_\frak{m},Y_\frak{m}\right\rangle \left\langle Y_\frak{m},X\right\rangle -\left\langle Y_\frak{m},X\right\rangle \left\langle Y_\frak{m},Y_\frak{m}\right\rangle  \right) \} \\
&=\dfrac{\left\langle Y_\frak{m},Y_\frak{m}\right\rangle }{\left\langle X,Y_\frak{m}\right\rangle ^3}\left\langle [Y,Z]_\frak{m},2\left\langle Y_\frak{m},X\right\rangle Y_\frak{m}-\left\langle Y_\frak{m},Y_\frak{m}\right\rangle X\right\rangle \\
&=\frac{F^3(Y_\frak{m})}{\left\langle Y_\frak{m},Y_\frak{m}\right\rangle} \left\langle [Y,Z]_\frak{m},\dfrac{2}{F(Y_\frak{m})}Y_\frak{m}-X\right\rangle .
\end{align*}
For $Y_\frak{m}\neq 0$ the inequality $\frac{F^3(Y_\frak{m})}{\left\langle Y_\frak{m},Y_\frak{m}\right\rangle ^2}\neq 0$ implies that for any $Z\in \mathfrak{m}$,
 $g_{Y_{\mathfrak{m}}}( [Y,Z]_{\mathfrak{m}},Y_{\mathfrak{m}}) =0 $ if and only if $\left\langle [Y,Z]_{\mathfrak{m}},\dfrac{2}{F(Y_{\mathfrak{m}})}Y_{\mathfrak{m}}-X\right\rangle =0$.
\end{proof}

A direct consequence of the above proposition is the following corollary.

\begin{cor}
Consider the assumption of the previous proposition. Then the vector $X$ is a geodesic vector of $(G/H,\left\langle ,\right\rangle )$ if and only if it is a geodesic vector of $(G/H,F)$.
\end{cor}

\begin{cor}
Suppose that $(G/H,F)$ is a homogeneous Kropina space as proposition \ref{geodesic vector of Kropina} such that $F$
is of Douglas type. Then a nonzero vector $Y \in \mathfrak{g}$ is a geodesic vector of $(G/H,F)$ if and only if it is
a geodesic vector of $(G/H,\left\langle ,\right\rangle )$.
\end{cor}
\begin{proof}
According to the theorem 3.2 of \cite{An-Deng Monatsh}, $F$ is of Douglas type if and only  if $X$ is orthogonal to $[\mathfrak{m},\mathfrak{m}]_{\mathfrak{m}}$. Therefore we have
\begin{equation*}
 \left\langle [Y,Z]_{\mathfrak{m}},\frac{2}{F(Y_{\mathfrak{m}})}Y_{\mathfrak{m}}-X\right\rangle =\left\langle [Y,Z]_{\mathfrak{m}},\frac{2}{F(Y_{\mathfrak{m}})}Y_{\mathfrak{m}}\right\rangle.
\end{equation*}
The above equation completes the proof.
\end{proof}
In \cite{Yan-Deng}, it is shown that on any homogeneous Randers space $(G/H,F)$, there is at least one homogeneous geodesic issuing from each origin point. In a similar way we show that the same result is true for any homogeneous Kropina space.
\begin{prop}
Suppose that $(M,F)$ is a homogeneous manifold with a Kropina metric $F$. Then there is at least one homogeneous geodesic issuing from each origin point.
\end{prop}
\begin{proof}
Let $G$ be a connected Lie group of isometries acting transitively on $M$, and $H$ be the isotropy group of a point $o \in M$. Suppose that $\textbf{\textsf{K}}$ and $\textsf{rad}\textbf{\textsf{K}}$ denote the Killing form and it's null space, respectively. We recall that the null space $\textsf{rad}\textbf{\textsf{K}}$ is a solvable ideal of $\mathfrak{g}$, so it is included in the radical of $\mathfrak{g}$ denoted by $\mathfrak{r}$ (see \cite{Humphreys}).\\
Let $\mathfrak{m}=\mathfrak{h}^{\bot}$. Then, since $\textbf{\textsf{K}}$ is nondegenerate on $\mathfrak{h}$, $\mathfrak{g}=\mathfrak{h}+\mathfrak{m}$ is a reductive decomposition and $\textsf{rad}\textbf{\textsf{K}} \subseteq \mathfrak{m}$ (see \cite{Kowalski-Szenthe}). Assume that $F$ is defined by a $G$-invariant Riemannian metric $\left\langle ,\right\rangle $ and a vector field $\tilde{X}$ such that $X=\tilde{X}(H)$. Also, we denote the corresponding scalar product on $\mathfrak{m}$  and the  $Ad(H)$-invariant vector on $\mathfrak{m}$  by $\left\langle ,\right\rangle $ and $X$, respectively.\\
There are two cases:
\begin{description}
  \item[Case 1 ($\textsf{rad}\textbf{\textsf{K}}=\mathfrak{m}$)] If $\textsf{rad}\textbf{\textsf{K}}=\mathfrak{m}$ then there is a solvable Lie group $L$ which acts transitively on $M$. Using the same argument as in the case $(A)$ of the proof of proposition $3$ in \cite{Kowalski-Szenthe} shows that there exists a reductive decomposition $\mathfrak{g}=\mathfrak{h}+\mathfrak{m}$ such that the $\mathfrak{m}$-projection $[\mathfrak{g}, \mathfrak{g}]_{\mathfrak{m}}$ is a proper subspace of $\mathfrak{m}$. Let $Y$ belongs to
      $ [\mathfrak{g}, \mathfrak{g}]_{\mathfrak{m}}^{\perp} $, with respect to $\left\langle ,\right\rangle $, such that $\left\langle Y,Y\right\rangle =1$. Thus, $W=\dfrac{1}{2}(\sqrt{\left\langle X,X\right\rangle }Y+X)$ is a geodesic vector. An easy computation shows that $F(W)=1$ and so for any $Z\in \mathfrak{m}$ we have $\left\langle [W,Z]_{\mathfrak{m}},\dfrac{2}{F(W_{\mathfrak{m}})}W_{\mathfrak{m}}-X\right\rangle =\left\langle [W,Z]_{\mathfrak{m}},\sqrt{\left\langle X,X\right\rangle }Y+X-X\right\rangle =0$. Therefore there exists one homogeneous geodesic through $o$.
\item[Case 2 ($\textsf{rad}\textbf{\textsf{K}}\subsetneq\mathfrak{m}$)] If $\textsf{rad}\textbf{\textsf{K}}$ is a proper subspace of $\mathfrak{m}$, the same argument as in the case $(2)$ of the proof of theorem $2.9$ in \cite{Yan-Deng} shows that $\mathfrak{m}$ has an eigenspace decomposition $\mathfrak{m}=V_0+V_1+...+V_s$ with respect to a $\textbf{\textsf{K}}$-symmetric endomorphism $\theta:\mathfrak{m} \longrightarrow \mathfrak{m}$ defined by $\textbf{\textsf{K}}(X,Y)=\left\langle \theta(X),Y\right\rangle $ such that $V_0=\textsf{rad}\textbf{\textsf{K}}$. Let $\{f_1,f_2,...,f_r\}$ be a $\left\langle ,\right\rangle $-orthonormal basis of $V=V_1+V_2+...+V_s$ such that $\theta(f_i)=\lambda _if_i$ for $i=1,2,...,r$. Suppose that $X=X_0+\Sigma_{i=1}^rx_if_i$ and $Y=Y_0+\Sigma_{i=1}^ry_if_i$, where $X_0,Y_0\in V_0$ and $x_i,y_i \in \mathbb{R}$. By proposition \ref{geodesic vector of Kropina}, $Y$ is a geodesic vector of $(M,F)$ if and only if the following equation equals zero,
\begin{eqnarray*}
    &&\left\langle [Y,Z]_{\mathfrak{m}},\frac{2}{F(Y_{\mathfrak{m}})}Y_{\mathfrak{m}}-X\right\rangle  \\
    && \hspace*{1cm}= \left\langle [Y,Z]_{\mathfrak{m}},\frac{2}{F(Y_{\mathfrak{m}})}(Y_0+\Sigma_{i=1}^ry_if_i)-X_0-\Sigma_{i=1}^rx_if_i\right\rangle\\
    && \hspace*{1cm}= \left\langle [Y,Z]_{\mathfrak{m}},\frac{2}{F(Y_{\mathfrak{m}})}Y_0-X_0\right\rangle +\left\langle [Y,Z]_{\mathfrak{m}},\frac{2}{F(Y_{\mathfrak{m}})}\Sigma_{i=1}^ry_if_i- \Sigma_{i=1}^rx_if_i\right\rangle\\
    && \hspace*{1cm}= \left\langle [Y,Z]_{\mathfrak{m}},\frac{2}{F(Y_{\mathfrak{m}})}Y_0-X_0\right\rangle +\left\langle [Y,Z]_{\mathfrak{m}},\frac{2}{F(Y_{\mathfrak{m}})}\Sigma_{i=1}^ry_i\frac{\theta(f_i)}{\lambda _i}- \Sigma_{i=1}^rx_i\frac{\theta(f_i)}{\lambda _i}\right\rangle \\
    && \hspace*{1cm}= \left\langle [Y,Z]_{\mathfrak{m}},\frac{2}{F(Y_{\mathfrak{m}})}Y_0-X_0\right\rangle +\textbf{\textsf{K}}\left([Y,Z]_{\mathfrak{m}},\frac{2}{F(Y_{\mathfrak{m}})}\Sigma_{i=1}^ry_i\frac{f_i}{\lambda _i}- \Sigma_{i=1}^rx_i\frac{f_i}{\lambda _i}\right)
\end{eqnarray*}
\begin{eqnarray*}
    && \hspace*{1cm}= \left\langle [Y,Z]_{\mathfrak{m}},\frac{2}{F(Y_{\mathfrak{m}})}Y_0-X_0\right\rangle +\textbf{\textsf{K}}\left(  [Y,Z],\frac{2}{F(Y_{\mathfrak{m}})}\Sigma_{i=1}^ry_i\frac{f_i}{\lambda _i}- \Sigma_{i=1}^rx_i\frac{f_i}{\lambda _i}\right)  \\
    && \hspace*{1cm}= \left\langle [Y,Z]_{\mathfrak{m}},\frac{2}{F(Y_{\mathfrak{m}})}Y_0-X_0\right\rangle +\textbf{\textsf{K}}\left(  Z,\left[ \frac{2}{F(Y_{\mathfrak{m}})}\Sigma_{i=1}^ry_i\frac{f_i}{\lambda _i}- \Sigma_{i=1}^rx_i\frac{f_i}{\lambda _i},Y\right] \right)  \\
    && \hspace*{1cm}= \left\langle [Y,Z]_{\mathfrak{m}},\frac{2}{F(Y_{\mathfrak{m}})}Y_0-X_0\right\rangle +\textbf{\textsf{K}}\left(  Z,\left[ \frac{2}{F(Y_{\mathfrak{m}})}\Sigma_{i=1}^ry_i\frac{f_i}{\lambda _i}- \Sigma_{i=1}^rx_i\frac{f_i}{\lambda _i},\Sigma_{i=1}^ry_if_i\right] \right).
\end{eqnarray*}

The above equation  equals zero if the following equations have a solution of the form
$(Y_0, y_1, y_2,\cdots, y_r,t)$:
\begin{equation}
 \left\{
\begin{array}{rl}
 & F(Y)=2,\\
& Y_0=X_0,\\
& \dfrac{y_i-x_i}{\lambda _i}=ty_i.
\end{array} \right.
\end{equation}
In the case $X=X_0$, we easily obtain the following solution:
\begin{equation*}
   Y_0=X_0, \qquad y_1=\sqrt{\left\langle X_0,X_0\right\rangle }, \qquad t=\dfrac{1}{\lambda _1}, \qquad y_i=0, i=2,...,r.
\end{equation*}
In the case $X\neq X_0$, without loss of generality, we can assume that $\vert \lambda _1\vert \geq \vert \lambda _2\vert \geq \cdots \geq \vert \lambda _r\vert >0$. Let $Y_0=X_0$, $y_i(t)=\dfrac{x_i}{1-t\lambda _i}$, $i=1,2,\cdots,r$, $Y(t)=X_0+\sum _{i=1}^ry_i(t)f_i$. We can easily see that $M(t)=F(Y(t))-2$ is a continues function on $(\frac{-1}{\vert \lambda _1\vert} , \frac{1}{\vert \lambda _1\vert})$ and $M(0)<0$ and $\lim _{t\rightarrow \frac{1}{\lambda _1}} M(t) =+\infty$. On the other hand for any $t\in(\frac{-1}{\vert \lambda _1\vert} , \frac{1}{\vert \lambda _1\vert})$ we have $F(Y(t))>0$. So the intermediate value theorem states that there exists $t=t_0$ such that $M(t_0)=0$ and $F(Y(t_0))>0$. We observe that $Y_0=X_0$,  $y_i=y_i(t_0)$ and $t=t_0$ is a solution and it completes the proof.
\end{description}
\end{proof}


\section{\textbf{Homogeneous Geodesics in Homogeneous $(\alpha,\beta)$-Spaces}}
In this section we investigate homogeneous geodesics on homogeneous Finsler spaces equipped with $(\alpha ,\beta)$-metrics and obtain interesting results.
\begin{prop}
Suppose that $(M=G/H,F)$ is a homogeneous Finslerian manifold and $F$ is an invariant $(\alpha , \beta )$-metric defined by an invariant Riemannian metric $\left\langle ,\right\rangle $ and an invariant vector field $\tilde{X}$ on $M$ such that $X=\tilde{X}(H)$ is a geodesic vector with respect to the Riemannian metric $\left\langle ,\right\rangle$. Then $(M,F)$ admits a homogeneous geodesic through any point.
\end{prop}
\begin{proof}
$X$ is a geodesic vector of the Riemannian manifold $(M,\left\langle ,\right\rangle)$. So, by theorem 2.2 of \cite{Salimi-Parhizkar}, $X$ is a geodesic vector of $(M,F)$. Hence at least one homogeneous geodesic through $o$ exists.
\end{proof}

The next proposition develops the main result of previous section to $(\alpha ,\beta)$-metrics with some conditions. We use some ideas from \cite{Kowalski-Szenthe} in the proof.

\begin{prop}
Let $(M=G/H,F)$ be a homogeneous Finslerian manifold. Suppose that $F$ is an invariant $(\alpha , \beta )$-metric, arisen from an invariant Riemannian metric $\left\langle ,\right\rangle $ and an invariant vector field $\tilde{X}$ on $M$. If  for any $Y\in \mathfrak{g}\setminus \{0\} $ and $Z \in \mathfrak{m}$ the equality $\left\langle X,[Y,Z]_{\mathfrak{m}}\right\rangle =0$ holds and moreover $\varphi'' (r_{\mathfrak{m}}) \leq 0$, where $r_{\mathfrak{m}}=\frac{\left\langle X,Y_{\mathfrak{m}}\right\rangle }{\sqrt{\left\langle Y_{\mathfrak{m}}, Y_{\mathfrak{m}}\right\rangle }}$, then either at least one homogeneous geodesic through any $o \in M$ exists or $M=G/H$ such that $G$ is a semi-simple group.
\end{prop}

\begin{proof}
The main idea is similar to that used in the proof of proposition 3 of \cite{Kowalski-Szenthe}. Assume that $G/H$ is an arbitrary representation of $M$, where $G$ is a transitive Lie group of isometries. If we put $o=H$ then there are two cases:
\begin{enumerate}
  \item There exists $i\geq 0$ such that $T_e \pi (\mathfrak{g}^{(i)})=T_oM$ and $T_e \pi (\mathfrak{g}^{(i+1)}) \subsetneq T_oM$,
  \item For every $i \geq 0$, $T_e \pi (\mathfrak{g}^{(i)})=T_oM$.
\end{enumerate}
In the case $(1)$, by attention to the part $(A)$ of the proof of proposition $3$ in \cite{Kowalski-Szenthe}, there exists a geodesic vector $Y\neq 0$ for the Riemannian manifold $(M,\left\langle ,\right\rangle)$. Now, theorem $2.3$ of \cite{Salimi-Parhizkar} shows that $Y$ is a geodesic vector of $(M,F)$. Thus, there exists at least one homogeneous geodesic through $o$. In the case $(2)$, by attention to the part $(B)$ of the proof of proposition $3$ in \cite{Kowalski-Szenthe}, the Lie group $G$ is semi-simple.
\end{proof}

Next, we discuss the existence of homogeneous geodesic through $o$ on $(M=G/H, F)$ where $G$ is a semi-simple group of isometries and $F$ is an invariant $(\alpha , \beta)$-metric induced by an invariant Riemannian metric $\left\langle ,\right\rangle$. For this, we consider the Killing form of $\mathfrak{g}$, denoted by $\textbf{\textsf{K}}$. We observe that $\textbf{\textsf{K}}$ is a nondegenerate form on $\mathfrak{g}$ and moreover, its restriction  to $\mathfrak{h}$ is nondegenerate. If $\mathfrak{g}=\mathfrak{m}+\mathfrak{h}$ is the reductive decomposition such that $\frak{m}$ is the orthogonal complement of $\frak{h}$ with respect to $\textbf{\textsf{K}}$, then $\textbf{\textsf{K}}$ is nondegenerate on $\mathfrak{m}$, too. Let $\left\langle  , \right\rangle $ be the corresponding scalar product on $\mathfrak{m}$ induced by the Riemannian metric $\left\langle ,\right\rangle$. Now we define an isomorphism $\theta$ of $\mathfrak{m}$  onto $\mathfrak{m}$ as follows:\\
\begin{equation}
\left\langle \theta (X),Y\right\rangle =\textbf{\textsf{K}}(X,Y), \qquad X,Y \in \mathfrak{m}.
\end{equation}
It is easy to show that the corresponding matrix of $\theta$ and the corresponding matrix of $\textbf{\textsf{K}}$ are equal with respect to a $\left\langle  ,\right\rangle$-orthonormal basis of $\mathfrak{m}$. So $\mathfrak{m}$ admits an $\left\langle  ,\right\rangle $-orthonormal basis $(f_1,\cdots,f_m)$ of  eigenvectors (for more details see \cite{Kowalski-Szenthe} ).

\begin{prop}
Suppose that $(M=G/H,F)$ is a homogeneous Finslerian manifold, dim$M=m$ and $G$ is a semi-simple group of isometries. Let  $F$ be a $(\alpha , \beta)$-metric which is defined by an invariant Riemannian metric $\left\langle ,\right\rangle $ and an invariant vector field $\tilde{X}$ on $M$. If  for any $Y\in \mathfrak{g}\setminus \{0\}$ and $Z \in \mathfrak{m}$ the equality $\left\langle X,[Y,Z]_{\mathfrak{m}}\right\rangle =0$ holds and moreover $\varphi'' (r_{\mathfrak{m}}) \leq 0$, where $r_{\mathfrak{m}}=\frac{\left\langle X,Y_{\mathfrak{m}}\right\rangle }{\sqrt{\left\langle Y_{\mathfrak{m}}, Y_{\mathfrak{m}}\right\rangle }}$, then $m$ mutually orthogonal homogeneous geodesics issuing from any origin exist.
\end{prop}
\begin{proof}
Likewise the proof of theorem 1 in \cite{Kowalski-Szenthe}, we can show that each eigenvector $f_i$ is a geodesic vector of $(M,\left\langle ,\right\rangle )$. So according to the theorem 2.3 in \cite{Salimi-Parhizkar}, it is a geodesic vector of $(M,F)$ and this completes the proof.
\end{proof}

\begin{prop}
Let $(M,F)$ be as in the previous proposition and $\rho\subseteq \mathfrak{g}$ satisfies the following conditions
\begin{itemize}
\item[a)] $Ad(H)(\rho) \subseteq \rho$,
\item[b)] $\rho$ is irreducible with respect to the restriction of adjoint representation to $H$,
\item[c)] $\textbf{\textsf{K}}$ is nondegenerate on $\rho$, and $\rho$ is $\textbf{\textsf{K}}$-orthogonal to $\mathfrak{h}$.
\end{itemize}
Then the nonzero elements of $\rho$ are geodesic vectors of $(M,F)$.
\end{prop}
\begin{proof}
According to the theorem 2 in \cite{Kowalski-Szenthe}, all nonzero element of $\rho$ are geodesic vectors with respect to $\left\langle ,\right\rangle $. So, using  the theorem 2.3 in \cite{Salimi-Parhizkar}, show that these elements are geodesic vectors of $(M,F)$.
\end{proof}

\begin{definition}
A homogeneous Riemannian manifold $G/H$ with an invariant Riemannian metric $\textbf{\emph{a}}$ is called a naturally reductive homogeneous space if there exists a reductive decomposition $\mathfrak{g}=\mathfrak{m} +\mathfrak{h}$ such that\\
\begin{equation}
\textbf{\emph{b}}(X,[Z,Y]_{\mathfrak{m}})+\textbf{\emph{b}}([Z,X]_{\mathfrak{m}},Y)=0 \qquad X,Y,Z \in \mathfrak{m},
\end{equation}
where $\textbf{\emph{b}}$ is the induced inner product on $ \mathfrak{m}$ by $\textbf{\emph{a}}$ and the subscript $\mathfrak{m}$ means the corresponding projection.
\end{definition}

In the literature, there are two different definitions of naturally reductive Finsler manifolds (see \cite{Deng-Hou-1} and \cite{Latifi}). One of them which is defined by Deng and Hou in \cite{Deng-Hou-1}, is as follows
\begin{definition}
Let $(M=G/H,F)$ be a  homogeneous Finsler space with an invariant Finsler metric. We call $(M=G/H,F)$ is a naturally reductive homogeneous space if there exists an invariant Riemannian metric $\textbf{\emph{a}}$ which $(M=G/H,\textbf{\emph{a}})$ is naturally reductive and the connections of $F$ and $\textbf{\emph{a}}$ coincide.
\end{definition}

\begin{cor}
Suppose that $(M=G/H,F)$ is a homogeneous Finslerian manifold such that $G$ is a semi-simple group of isometries and  $F$ is a $(\alpha ,\beta )$-metric arisen from an invariant Riemannian metric $\left\langle ,\right\rangle $ and an invariant vector field $\tilde{X}$ on $M$. Let  for any $Y\in \mathfrak{g}\setminus \{0\} $ and $Z \in \mathfrak{m}$ the equality $\left\langle X,[Y,Z]_{\mathfrak{m}}\right\rangle =0$ holds and moreover $\varphi'' (r_{\mathfrak{m}}) \leq 0$, where $r_{\mathfrak{m}}=\frac{\left\langle X,Y_{\mathfrak{m}}\right\rangle }{\sqrt{\left\langle Y_{\mathfrak{m}}, Y_{\mathfrak{m}}\right\rangle }}$. If $(M,\left\langle ,\right\rangle )$ is an isotropy irreducible homogeneous manifold, then $(M,F)$ is naturally reductive.
\end{cor}
\begin{proof}
By theorem 2 in \cite{Kowalski-Szenthe}, all element belonging to $\mathfrak{m}$ are geodesic vectors of $(M,\left\langle ,\right\rangle )$. Hence with the assumptions and notations of theorem 2.3 in \cite{Salimi-Parhizkar}, these are  geodesic vectors of $(M,F)$ and so by theorem 3.1 in \cite{Deng-Hou-3} $(M,F)$  is naturally reductive.
\end{proof}

In the end we discuss concerning homogeneous geodesic vectors on  3-dimensional non-unimodular Lie groups equipped with invariant Randers metric of Douglas type.
Suppose that $G$ is a 3-dimensional non-unimodular Lie group with Lie algebra $\mathfrak{g}$ equipped with a left invariant Riemannian metric $\left\langle ,\right\rangle $. Then, there exists an orthonormal basis $\{e_1,e_2,e_3\}$ of $\mathfrak{g}$ such that the bracket is expressed as
\begin{align*}
&[e_1,e_2]=\alpha e_2+\beta e_3, \\
&[e_1,e_3]=\gamma e_2+\delta e_3, \\
&[e_2,e_3]=0,
\end{align*}
where $\alpha$, $\beta$, $\gamma$, $\delta$ are real numbers such that $\alpha + \delta \neq 0$ and $\alpha \gamma +\beta \delta =0$. This basis also diagonalizes the Ricci form (for more details see \cite{Milnor}). \\

Denote $D=( \beta + \gamma )^2-4 \alpha \delta$ such that all Ricci eigenvalues are distinct. Then, up to a reparametrization, the space $(G,\left\langle ,\right\rangle )$ admits just one or just two or just three homogeneous  geodesic through a point if $D < 0$ or $D=0$ or $D > 0$, respectively (see \cite{Kowalski-Nikcevic-Vlasek}, proposition 3.2). Here we generalize this result as follow
\begin{prop}
Let $(G,F)$ be a 3-dimensional non-unimodular Lie group with a Randers metric of Douglas type induced by a left invariant Reimannian metric $\left\langle ,\right\rangle $ and a left invariant vector field $\tilde{X}$. If $D$ is as above then, up to a reparametrization, $(G,F)$ admits  just one or just two or just three homogeneous  geodesics through a point for $D < 0$ or $D=0$ or $D > 0$, respectively. For the case $D=0$ they are mutually orthogonal and for $D>0$, they are linearly independent but never mutually orthogonal.
\end{prop}
\begin{proof}
Using the corollary 2.7 in \cite{Yan-Deng}, $U$ is a geodesic vector of $(G,F)$ if and only if it is a geodesic vector of $(G,\left\langle ,\right\rangle )$. So we can use the same argument as in the proof of proposition 3.2 of \cite{Kowalski-Nikcevic-Vlasek} and prove this result.
\end{proof}

{\large{\textbf{Acknowledgment.}}} We are grateful to the office of Graduate Studies of the University of Isfahan for their support.

\end{document}